\documentclass{article}
\usepackage{psfrag}
\usepackage[parfill]{parskip}
\usepackage{float}
\usepackage{graphicx}
\usepackage{epsfig}
\usepackage{amsmath, amsthm, amssymb}
\usepackage{verbatim}
\usepackage{multicol}
\usepackage{url}
\usepackage{latexsym}

\usepackage{mathrsfs}
\usepackage[colorlinks = true,
            linkcolor  = blue,
            citecolor  =blue,
            urlcolor   =blue,
            bookmarks  = true]{hyperref}
 \usepackage{enumitem}
\setlist[enumerate,1]{label=(\roman*)}
\usepackage[normalem]{ulem}
\usepackage{bm}
\usepackage{dsfont}
\usepackage{stmaryrd}
\usepackage[dvipsnames]{xcolor}  % 支持更多命名颜色（如 RoyalBlue）

\newcommand{\address}[1]{\begin{center}\small #1\end{center}}
\usepackage[backend=biber, style=numeric, citestyle=numeric-comp]{biblatex}
\DeclareFieldFormat{labelnumberwidth}{\mkbibbrackets{#1}}
\DeclareNameAlias{author}{family-given}

\DeclareFieldFormat*{title}{#1}

\AtEveryBibitem{%
  \clearfield{issn}
  \clearfield{url}
  \clearfield{eprint}
}

\renewbibmacro{in:}{}

\DeclareFieldFormat[article]{volume}{\textbf{#1}}
\DeclareFieldFormat[article]{number}{\textbf{#1}}
\renewbibmacro*{volume+number+eid}{%
  \printfield{volume}% 卷号
  \iffieldundef{number}
    {}%
    {\printtext[parens]{\printfield{number}}}% 圆括号包住期号
  \setunit{\addcomma\space}%
  \printfield{eid}}

\DeclareFieldFormat[journal]{journaltitle}{\mkbibemph{#1}}

\DeclareNameAlias{author}{family-given}
\DeclareNameAlias{editor}{family-given}

\DeclareDelimFormat{namedelim}{\addcomma\space}
\DeclareFieldFormat[article,incollection,inproceedings,book]{giveninit}{\mkbibnamegiven{#1}}

\DeclareNameFormat{family-given}{%
  \usebibmacro{name:family-given}
    {\namepartfamily}
    {\namepartgiveni} % 注意这里，用 i (initial缩写)
    {\namepartprefix}
    {\namepartsuffix}%
  \usebibmacro{name:andothers}}
\DeclareFieldFormat{pages}{#1}

\numberwithin{equation}{section}

\theoremstyle{plain}

\newtheorem{myrem}{Remark}[section]
\newtheorem{Theorem}{Theorem}[section]
\newtheorem{lemma}{Lemma}[section]
\newtheorem{proposition}{Proposition}[section]

\def\author#1{\par
    {\centering{\authorfont#1}\par\vspace*{0.05in}}
}

\def\titlefont{\fontsize{13}{15}\bfseries\boldmath\selectfont\centering{}}
\def\authorfont{\fontsize{13}{15}}

\let\affiliationfont\rhfont

\def\address#1{\par
    {\centering{\affiliationfont#1\par}}\par\vspace*{11pt}
}
\def\title#1{
    \thispagestyle{plain}
    \vspace*{-14pt}
    \vskip 79pt
    {\centering{\titlefont #1\par}}%
    \vskip 1em
}

\begin{document}

\title{The smallest singular value of
sparse discrete random matrices}

\vspace{1cm}

\author{Kexin Yu}
\address{Shandong University\\Kexinyu01@gmail.com}

\vspace{0.3cm}

\begin{abstract}
Let $M_n$ be an $n \times n$ random matrix with i.i.d. sparse discrete entries. In this paper, we develop a simple framework to solve the approximate Spielman-Teng theorem for $M_n$, which has the following form: There exist constants $C, c>0$ such that for all $\eta \geq 0$, $\textbf{P}\left(s_n\left(M_n\right) \leq \eta\right) \lesssim n^C \eta+\exp \left(-n^c\right)$.  As an application, we give an approximate Spielman-Teng theorem for $M_n$ whose entries are $\mu-$lazy random variables, extending previous work by Tao and Vu.
\end{abstract}

\section{Introduction and Main Results}\label{sec1}
Let $A_n$ be an $n\times n$ real matrix with singular values $s_1(A_n)\ge\cdots\ge s_n(A_n),$ in particular, the largest and smallest singular values of $A_n$ are defined to be 
\begin{align*}
s_1(A_n) = \sup_{x\in \mathbb{S}^{n-1}}\|A_n x\|_2, \quad s_n(A_n) = \inf_{x\in \mathbb{S}^{n-1}}\|A_n x\|_2,
\end{align*}
where $\|\cdot\|_2$ denotes the Euclidean norm on $\mathbb{R}^n$, and $\mathbb{S}^{n-1}$ denotes the $n-1$ dimensional Euclidean sphere in $\mathbb{R}^n$.
\par
In computer science, the condition number of $A_n$, i.e., the ratio of $s_1(A_n)$ and $s_n(A_n)$, is an important indicator to measure the stability of a linear system $A_n x=b$ to input perturbations. The larger the condition number, the more unstable the system is, and the error of the numerical solution will be magnified. Therefore, the condition number is widely used in many fields, such as algorithm stability analysis, numerical linear algebra, machine learning, and signal processing.
\par
Bounding the condition number of $A_n$ is an important problem in this field. The behavior of $s_1(A_n)$ is well understood under general assumptions on matrix entries $\xi$. In particular, Yin, Bai and Krishnaiah \cite{MR950344} proved that, when $\mathrm{E} \xi=0$ and $\mathrm{E}|\xi|^4<\infty$, with high probability
\begin{align*}
s_1(A_n) \sim \sqrt{n}.
\end{align*}
Hence, the main problem in estimating the condition number's magnitude is exploring the lower bound of $s_n(A_n)$.
\par
For Gaussian random matrices with i.i.d. $N(0,1)$ entries, the magnitude of $s_n(A)$ is of the order $1 / \sqrt{n}$ with high probability. This observation goes back to von Neumann \cite{MR157875}, and it was proved by Edelman \cite{doi:10.1137/0609045} that
\begin{align*}
\textbf{P}\left(s_n\left(A_n\right) \leq \varepsilon n^{-1 / 2}\right) \sim \varepsilon.
\end{align*}

For Rademacher random matrices with i.i.d. $\{\pm1\}$ entries, Spielman and Teng \cite{MR1989210} conjectured in the 2002 ICM survey that the same tail estimate (up to an exponentially small correction) should hold for Rademacher random matrices. This conjecture has since become known as the Spielman-Teng conjecture. It has motivated a great deal of work on the smallest singular value of random matrices.

The breakthrough work of Rudelson \cite{MR2434885} gave the first bound on the smallest singular value beyond the Gaussian case. He proved that if $A_n$ has i.i.d. mean–zero, unit–variance subgaussian entries, then
\begin{align*}
\textbf{P}\left(s_n\left(A_n\right) \leq \varepsilon n^{-3 / 2}\right) \lesssim \varepsilon+n^{-3 / 2}.
\end{align*}
Shortly thereafter, Tao and Vu \cite{MR2480613} developed the inverse Littlewood-Offord theory, which, as an important application, gave a new bound on the
smallest singular value of Bernoulli random matrices with i.i.d. $0/1$ entries: for any $A>0$, there exists $B>0$ such that,  
\begin{align*}
\textbf{P}\bigl(s_n(A_n)\le n^{-B}\bigr)\le n^{-A}.
\end{align*}
Based on the Littlewood-Offord problem, Rudelson and Vershynin \cite{MR2407948} proved a refined result; they gave
\begin{align*}
\textbf{P}\left(s_n\left(A_n\right) \leq \varepsilon n^{-1 / 2}\right) \leq C \varepsilon+e^{-c n},
\end{align*}
 for all random matrices with i.i.d. mean–zero, unit–variance subgaussian entries. This perspective was developed further in the 2020 breakthrough work of Tikhomirov \cite{MR4076632}, for Bernoulli random matrices, he gave more refined "constants",
\begin{align*}
\textbf{P}\left(s_n\left(A_n\right) \leq \varepsilon n^{-1 / 2}\right) \leq C \varepsilon+\left(\frac{1}{2}+o_n(1)\right)^n.
\end{align*}

\par
Recently, Ferber, Jain, Luh and Samotij \cite{MR4273471} developed a new approach to the counting problem in the inverse Littlewood-Offord theory, and then Jain \cite{MR4282089} gave a novel framework for finding probabilistic lower bounds on the smallest singular value of a discrete random matrix in a truly combinatorial setting. 

Inspired by previous work, this paper develops the framework described above for the study of the smallest singular value of a sparse discrete random matrix.
For any $0 \leq \mu \leq 1$, let $\eta^\mu \in\{-1,0,1\}$ denote a random variable equal to 0 with probability $1-\mu$ and $\pm 1$ with probability $\mu / 2$ each. This type of random variable is known as $\mu-$lazy random variable. A $\mu-$lazy random matrix is a matrix whose entries are i.i.d. $\mu-$lazy random variables. When $\mu$ is very small, the $\mu-$lazy random matrix is an essential example of sparse discrete random matrices. Sparse random matrices are important for the study of combinatorial random matrices. 
We obtain the small ball probability estimate for a $\mu-$lazy random matrix, and our method can be easily extended to general sparse discrete random matrices. In addition, the inverse Littlewood-Offord theorem tells us that when linear combinations of random variables are highly concentrated, their coefficient vectors must have some additive structure. Tao and Vu \cite{MR2480613}, in order to conveniently study the relationship between the concentration function and the vector structure of the random walks as described above, consider a fixed $\mu$ that can take values in $[0,1]$. However, in sparse conditions, many vectors have very few nonzero entries. Thus, in this paper, we consider a model in which the value of $\mu$ can tend to 0 as an extension of the study of Tao and Vu \cite{MR2480613}. In the remainder of this paper, $\mu_n$ means that $\mu$ is no longer fixed but converges to $0$ as $n$ increases.
\par
Specifically, we define the largest atom probability as follows.
\par
For an integer vector $\boldsymbol{w} \in \mathbb{Z}^n$, we define its largest atom probability to be
\begin{align*}
\rho^{(\mu)}(\boldsymbol{w}):=\sup _{x \in \mathbb{Z}} \textbf{P}\left(\eta^\mu_1 w_1+\cdots+\eta^\mu_n w_n=x\right),
\end{align*}
where $\eta^\mu_1, \ldots, \eta^\mu_n$ are i.i.d. $\mu-$lazy random variables.
\par
Let $Y_\mu(\boldsymbol{w}) := \sum_{i=1}^n\eta_i^{\mu}w_i$, then $Y_\mu(\boldsymbol{w})$ is a class of discrete random walks in integer space $\mathbb{Z}^n$, which start at the origin and consist of $n$ steps, where at the $i^{\text {th }}$ step one moves backwards or forwards with magnitude $v_i$ and probability $\mu / 2$, and stays at rest with probability $1-\mu$.
\par
Notice that the probability $\textbf{P}\left(\eta^\mu_1 w_1+\cdots+\eta^\mu_n w_n=0\right)$ is closely connected to the singularity probability of the matrix $M_n$, where $M_n$ is defined as an $n \times n$ random matrix with i.i.d. $\mu-$lazy entries. Since for any $\boldsymbol{w}=(w_1,...,w_n)$ (which we interpret as a column
vector), the entries in the product $M_n\boldsymbol{w}$ are independent copies of $Y_\mu(\boldsymbol{w})$. It follows that for such a row-independent matrix $M_n$, the study of its smallest singular value is an important application of the study of the inverse Littlewood-Offord theorem. Based on this, the paper obtained the following main results.

\begin{Theorem}\label{main}
Let $M_n$ denote an $n \times n$ random matrix, each of whose entries is an i.i.d. $\mu_n-$lazy random variable. Then, for any $\eta \geq 2^{-n^{0.0001}}$, $\mu_n\geq n^{-0.45}$,
\begin{align*}
\textbf{P}\left(s_n\left(M_n\right) \leq \eta\right) \lesssim \frac{\eta n^{3 / 2}}{\sqrt{\mu_n}}.
\end{align*}
\end{Theorem}
\begin{myrem}
We have not tried to optimize the constant 0.0001 in the above theorem, and similarly, for $\mu_n$, which characterizes the sparsity of sparse random matrices, it can be made more refined using an approach similar to that of this paper to make it closer to (but not lower than) $\frac{1}{\sqrt{n}}$.
\end{myrem}

\begin{myrem}
Since few nonzero elements exist for sparse random linear combinations, studying the famous inverse Littlewood–Offord theory is challenging. In this paper, we not only show that the inverse Littlewood-Offord theory under a sparse $\mu_n-$ lazy distribution retains the core conclusion of "high concentration $\Rightarrow$ strong structure" but also reveal how sparsity $\mu_n$ affects the small ball probability and the smallest singular value.
\end{myrem}

The remainder of this paper is organized as follows. Section \ref{sec2} first recalls some basic concepts and lists some concentration inequalities for classical random variables and extends them to general $\mu-$lazy random variables. Section \ref{sec3} gives the proof of the main results, and Section \ref{sec4} presents a discussion of the main approach and future work of this paper.

Throughout the paper, we use the notation $[n]:= \{1, \dots, n\}$. We will also use the following asymptotic notation: If $a = \Omega(b)$, we write $b \lesssim a$ or $a \gtrsim b$; If $a = o(b)$, we write $b \gg a$ or $a \ll b$.

\section{Preliminaries}\label{sec2}
In this section, we introduce some lemmas that will be used to prove our
main result. The main purpose is to extend some classical inequalities to the sparse case.
\par
First, we introduce an important notion defined by Rudelson and Vershynin. The least common denominator (LCD) of a nonzero vector $\boldsymbol{a} \in \mathbb{R}^n$ can be defined as
\begin{align*}
\operatorname{LCD}_{\gamma, \alpha}(\boldsymbol{a}):=\inf \left\{\theta>0: \operatorname{dist}\left(\theta \boldsymbol{a}, \mathbb{Z}^n\right) < \min \left\{\gamma\|\theta \boldsymbol{a}\|_2, \alpha\right\}\right\},
\end{align*}
for some parameters $\gamma \in(0,1)$ and $\alpha>0$.
\par
The following lemma, which appears in \cite{MR2407948}, shows how the LCD of a vector acts on its L\'{e}vy concentration function. Specifically, it shows that vectors with large LCD have small L\'{e}vy concentration function on scales which are larger than $\Omega(1 / \mathrm{LCD})$. Here, we reproduce a particularly simple proof for the $\mu-$lazy case from the lecture notes \cite{rudelson2013recentdevelopmentsnonasymptotictheory};
\begin{lemma}\label{lcd}
Let $\xi_1, \ldots, \xi_n$ denote i.i.d. $\mu-$lazy random variables, $0<\mu_{n}\leq1/2$. Consider a unit vector $\boldsymbol{a}=\left(a_1, \ldots, a_n\right) \in \mathbb{S}^{n-1}$. Let
\begin{align*}
S:=\sum_{i=1}^n \xi_i a_i.
\end{align*}
Then, for every $\alpha>0$, and for
\begin{align*}
\delta \geq \frac{(2 / \pi)}{\operatorname{LCD}_{\gamma, \alpha}(\boldsymbol{a})},
\end{align*}
we have
\begin{align*}
\mathcal{L}(S, \delta) \lesssim \frac{\delta}{\sqrt{\mu_{n}}\gamma}+\exp \left(-2\mu_{n}\alpha^2\right).
\end{align*}
\end{lemma}
\begin{proof}
Using Ess\'{e}en's inequality\footnote{It can be found in \cite{MR205297}.}, we can use the characteristic function to bound L\'{e}vy concentration function:
\begin{align*}
\mathcal{L}(X, 1) \lesssim \int_{-2}^2|\textbf{E}[\exp (i \theta X)]| d \theta.
\end{align*}
Then, we have
\begin{align*}
\mathcal{L}(S, \delta) & =\mathcal{L}(S / \delta, 1) \\
& \lesssim \int_{-2}^2|\textbf{E}[\exp (i \theta S / \delta)]| d \theta \\
& =\int_{-2}^2 \prod_{j=1}^n\left|\textbf{E}\left[\exp \left(i a_j \xi_j \theta / \delta\right)\right]\right| d \theta \\
& =\int_{-2}^2 \prod_{j=1}^n\left|1-\mu_{n}+\mu_{n}\cos \left(a_j \theta / \delta\right)\right| d \theta \\
& \leq \int_{-2}^2 \prod_{j=1}^n \exp \left(-2\mu_{n}\sin ^2\left(a_j \theta / 2\delta\right)\right) d \theta \\
& \leq \int_{-2}^2 \prod_{j=1}^n \exp \left(-2\mu_{n} \min _{q \in \mathbb{Z}}\left|\frac{ \theta}{\pi \delta} a_j-q\right|^2\right) d \theta,
\end{align*}
where the second inequality uses a deflation similar to that of (\ref{Hal6}), and in the last line, we have used the inequality $|\sin (x)| \geq \min _{q \in \mathbb{Z}}\left|\frac{2}{\pi} x-q\right|$. Then, let
\begin{align*}
h(\theta):=\min _{\boldsymbol{p} \in \mathbb{Z}^n}\left\|\frac{ \theta}{\pi \delta} \boldsymbol{a}-\boldsymbol{p}\right\|_2,
\end{align*}
we have
\begin{align*}
\mathcal{L}(S, \delta) \lesssim \int_{-2}^2 \exp \left(-2\mu_{n} h^2(\theta)\right) d \theta.
\end{align*}
for any $\theta \in[-2,2]$, let $s=\frac{\theta}{\pi\delta}\in[-\frac{2}{\pi\delta},\frac{2}{\pi\delta}]$, since, by assumption, $2 /(\pi \delta) \leq \operatorname{LCD}_{\gamma, \alpha}(\boldsymbol{a})$, we obtain
\begin{align*}
h(\theta)= \min _{\boldsymbol{p} \in \mathbb{Z}^n}\left\|s \boldsymbol{a}-\boldsymbol{p}\right\|_2\geq \min \left(\gamma s\|\boldsymbol{a}\|_2, \alpha\right)=\min \left(\gamma s, \alpha\right),
\end{align*}
so that
\begin{align*}
\mathcal{L}(S, \delta) & \lesssim \int_{-2}^2\left(\exp \left(-2\mu_{n}( \gamma \theta / \pi \delta)^2\right)+\exp \left(-2\mu_{n}\alpha^2\right)\right) d \theta \\
& \lesssim \frac{\delta}{\sqrt{\mu_{n}}\gamma}+\exp \left(-2\mu_{n}\alpha^2\right).
\end{align*}
\end{proof}

The next inequality given is an extension of Hal\'{a}sz's inequality \cite{MR494478} on $\mathbb{F}_p$ with respect to $\mu_n-$lazy random variables. Before the proof, we need to introduce a quantity that describes the structure of the vector.
\par
Suppose that $\boldsymbol{a} \in \mathbb{F}_p^n$ for an integer $n$ and a prime $p$ and let $k \in \mathbb{N}$. For every $\beta \in[0,1]$, we define $R_k^\beta(a)$ as the number of solutions to
\begin{align*}
\pm a_{i_1} \pm a_{i_2} \cdots \pm a_{i_{2 k}}=0 \quad \bmod p,
\end{align*}
that satisfy $\left|\left\{i_1, \ldots, i_{2 k}\right\}\right| \geqslant(1+\beta) k$. Then we let $R_k(\boldsymbol{a})$ denote the number of solutions to $\pm a_{i_1} \pm a_{i_2} \cdots \pm a_{i_{2 k}}=0 \bmod p$, where repetitions are allowed in the choice of $\left\{i_1, \ldots, i_{2 k}\right\}\in[n]$, it is easily seen that $R_k(\boldsymbol{a})$ cannot be much larger than $R_k^\beta(\boldsymbol{a})$. This is formalized in the following simple lemma.

\begin{lemma}[Lemma 1.6 in \cite{MR4273471}]\label{Ra1}
For all integers $k$ and $n$ with $k \leqslant n / 2$, any prime $p$, vector $\boldsymbol{a}:=\left(a_1, \ldots, a_n\right) \in \mathbb{F}_p^n$, and $\beta \in$ $[0,1]$ we have
\begin{align*}
R_k(\boldsymbol{a}) \leqslant R_k^\beta(\boldsymbol{a})+\left(40 k^{1-\beta} n^{1+\beta}\right)^k.
\end{align*}
\end{lemma}
In the remainder of the paper we define $R_k^*(\boldsymbol{a})=R_k^{0.01}(\boldsymbol{a})$.
\begin{Theorem}\label{Hal}
There exists a constant $C$ such that the following holds for every odd prime $p$, integer $n$, and vector $\boldsymbol{a}:=\left(a_1, \ldots, a_n\right) \in \mathbb{F}_p^n \backslash\{\mathbf{0}\}$. Suppose that an integer $0 \leq k \leq n / 2$ and positive real $M$ satisfy
\begin{align*}
30 M \leq|\operatorname{supp}(\boldsymbol{a})| \quad \text { and } \quad 80 k M \leq n .
\end{align*}
Then,
\begin{align}\label{inverse1}
\rho^{(\mu_{n})}_{\mathbb{F}_p}(\boldsymbol{a}) \leq \frac{1}{p}+\frac{C R_k^*(\boldsymbol{a})+C\left(40 k^{0.99} n^{1.01}\right)^k}{2^{2 k} n^{2 k} \sqrt{\mu_{n} M}}+e^{-16\mu_{n} M} .
\end{align}
Where, $\rho^{(\mu_{n})}_{\mathbb{F}_p}(\boldsymbol{a})$ denotes the largest atom probability of $\boldsymbol{a}$ over $\mathbb{F}_p$, $0<\mu_{n}\leq1/2$.
\end{Theorem}
\begin{myrem}
Comparing with the non-sparse form with $\mu_n=1$, although the second term on the right-hand side of the (\ref{inverse1}) needs to be multiplied by $\frac{1}{\sqrt{\mu_n}}$, note that if $\rho^{(\mu_{n})}_{\mathbb{F}_p}(\boldsymbol{a})$ is large, then $R_k^*(\boldsymbol{a})$ must be large. Thus, we have the following conclusion:
\par
Although sparsity has some effect, $\mu_n-$lazy random variables still have the key conclusion of "high concentration $\Rightarrow$ strong structure". Thus, this Hal\'{a}sz’s type inequality may be viewed as a partial inverse Littlewood–Offord theorem. 
\end{myrem}

\begin{myrem}
With the division of the LCD and the study of the corresponding L\'{e}vy concentration function, we can transform the study of vectors on the unit sphere $\mathbb{S}^{n-1}$ into integer vectors, and furthermore, with the division of the sparsity of integer vectors and the study of the Hal\'{a}sz’s type inequality, we can obtain the final result.
\end{myrem}
\begin{proof}
Let $e_p$ be the canonical generator of the Pontryagin dual of $\mathbb{F}_p$,  $e_p:\mathbb{F}_p \rightarrow \mathbb{C}$ be defined by $e_p(x)=\exp(2\pi ix/p)$. The following discrete Fourier identity in $\mathbb{F}_p$
\begin{align*}
  \delta_0(x)=\frac1p\sum_{r\in\mathbb{F}_p}e_p(r x),  \quad
  \delta_0(x)=
  \begin{cases}
    1,&x=0\\
    0,&x\neq0
  \end{cases} .
\end{align*}
Let $\xi_1,\dots,\xi_n$ be i.i.d. $\mu$-lazy random variables, for any $q\in\mathbb{F}_p$,
\begin{align*}
\textbf{P}\left(\sum_{j=1}^n\xi_ja_j = q\right)
&= \textbf{E}\left[\delta_0\left(\sum_{j=1}^n\xi_ja_j - q\right)\right]\\
& = \textbf{E}\left[\frac1p\sum_{r\in\mathbb{F}_p}e_p\left(r(\sum_{j=1}^n\xi_ja_j - q)\right)\right]\\
&= \frac1p\sum_{r\in\mathbb{F}_p}e_p(-r q)\;
     \textbf{E}\left[\prod_{j=1}^n e_p(r\xi_ja_j)\right]\\
&= \frac1p\sum_{r\in\mathbb{F}_p}e_p(-r q)\,
\prod_{j=1}^n \textbf{E}\left[e_p(r\xi_ja_j)\right].
\end{align*}
Since
\begin{align*}
  \textbf{E}\left[e_p(r\xi_ja_j)\right]
  =1-\mu_{n}
   +\frac{\mu_{n}}{2}e^{2\pi ir a_j/p}
   +\frac{\mu_{n}}{2}e^{-2\pi ir a_j/p}
  =1-\mu_{n}+\mu_{n}\cos(\frac{2\pi ra_j}{p}),
\end{align*}
we obtain
\begin{align*}
  \textbf{P}\left(\sum_{j=1}^n\xi_ja_j=q\right)
  &=\frac1p\sum_{r\in\mathbb{F}_p}e_p(-r q)
    \prod_{j=1}^n\left(1-\mu_{n}+\mu_{n}\cos(\frac{2\pi ra_j}{p})\right)\\
  &\le \frac1p\sum_{r\in\mathbb{F}_p}
    \prod_{j=1}^n\left|1-\mu_{n}+\mu_{n}\cos(\frac{2\pi ra_j}{p})\right|.
\end{align*}
where, since $0<\mu_{n}\le1/2$,
\begin{align}\label{Hal6}
\left|1-\mu_{n}+\mu_{n}\cos(\pi x)\right|
=1-2\mu_{n}\sin^2(\frac {\pi x}2)
\le\exp\left(-2\mu_{n}\sin^2(\frac {\pi x}2)\right),
\end{align}
we have 
\begin{align}\label{Hal1}
  \textbf{P}\left(\sum_{j=1}^n\xi_ja_j=q\right)\le \frac1p\sum_{r\in\mathbb{F}_p}
    \prod_{j=1}^n\exp\left(-2\mu_{n}\sin^2(\frac {\pi ra_j}p)\right).
\end{align}
Given a real number $y$, denote by $\min_{q\in \mathbb{Z}}|y-q| \in[0,1 / 2]$ the distance between $y$ and a nearest integer. Let us recall the useful inequality
\begin{align*}
\sin^2\left(\pi y\right) \ge 4\min_{q\in \mathbb{Z}}|y-q|^2.
\end{align*}
Using this inequality to bound each term on the right-hand side of (\ref{Hal1}), we have
\begin{align}\label{Hal2}
  \textbf{P}\left(\sum_{j=1}^n\xi_ja_j=q\right)\le \frac1p\sum_{r\in\mathbb{F}_p}
   \exp\left(-8\mu_{n} \sum_{j=1}^n\min_{q\in \mathbb{Z}}|\frac { ra_j}p-q|^2\right).
\end{align}
Define the  level set for each non-negative real $t$ as follows\footnote {It is introduced in Ferber, Jain, Luh and Samotij \cite{MR4273471}.}:
\begin{align*}
T_t:=\left\{r \in \mathbb{F}_p: \sum_{j=1}^n\min_{q\in \mathbb{Z}}|\frac { ra_j}p-q|^2 \leqslant t\right\}.
\end{align*}

Thus, the right-hand side of (\ref{Hal2}) can be written as
\begin{align}\label{Hal3}
\sum_{r \in \mathbb{F}_p} \exp \left(-8\mu_{n} \sum_{j=1}^n\min_{q\in \mathbb{Z}}|\frac { ra_j}p-q|^2\right)&=\sum_{r \in \mathbb{F}_p} \int_0^{\infty} \mathbb{I}_{\left[\sum_{j=1}^n\min_{q\in \mathbb{Z}}|\frac { ra_j}p-q|^2 \leqslant \frac{t}{8\mu_{n}}\right]} e^{-t}dt\\
&=\int_0^{\infty}\left|T_{\frac{t}{8\mu_{n}}}\right| e^{-t} dt.
\end{align}

In fact, due to Ferber, Jain, Luh and Samotij \cite{MR4273471},
if $t \leqslant 2 M \leqslant|\operatorname{supp}(\boldsymbol{a})| / 15$, 
\begin{align}\label{Hal4}
\left|T_t\right| \leqslant \frac{\sqrt{2 t} \cdot\left|T_{2 M}\right|}{\sqrt{M}}+1, ~\left|T_{2M}\right| \leqslant \frac{\sqrt{2} p R_k(\boldsymbol{a})}{2^{2 k} n^{2 k}}
\end{align}
hold, 
where $R_k(\boldsymbol{a})$ is defined as above. Combining (\ref{Hal2})-(\ref{Hal4}), we obtain
\begin{align*}
\rho^{(\mu_{n})}_{\mathbb{F}_p}(\boldsymbol{a}) &=
\max _{q \in \mathbb{F}_p} \textbf{P}\left(\sum_{j=1}^n \xi_j a_j=q\right)\\ & \leqslant \frac{1}{p} \int_0^{16\mu_{n} M}\left|T_{\frac{t}{8\mu_{n}}}\right| e^{-t} d t+\frac{1}{p} \int_{16\mu_{n} M}^{\infty} p e^{-t} d t \\
& \leqslant \frac{1}{p} \int_0^{16\mu_{n} M}\left(\frac{\sqrt{ t} \cdot\left|T_{2 M}\right|}{2\sqrt{\mu_{n} M}}+1\right) e^{-t} d t+e^{-16\mu_{n} M} \\
& \leqslant \frac{\left|T_{2 M}\right|}{2p \sqrt{\mu_{n} M}} \cdot \int_0^{16\mu_{n} M} \sqrt{t} e^{-t} d t+\frac{1}{p} \int_0^{16\mu_{n} M} e^{-t} d t+e^{-16\mu_{n} M} \\
& \leqslant \frac{\left|T_{2 M}\right|}{2p \sqrt{\mu_{n} M}} \cdot \frac{\sqrt{\pi}}{2}+\frac{1}{p}+e^{-16\mu_{n} M} \\
& \leqslant \frac{C R_k(\boldsymbol{a})}{2^{2 k} n^{2 k} \sqrt{\mu_{n} M}}+\frac{1}{p}+e^{-16\mu_{n} M}.
\end{align*}
\par
Finally, using Lemma \ref{Ra1} and setting $\beta = 0.01$ to complete the proof.
\end{proof}
\par
The following lemma is a direct consequence of Theorem 1.7 in \cite{MR4282089}. It tells us the upper bound on the number of vectors over $\mathbb{F}_p^n$, which do not have a 'large' subvector with'small' $R_k^*$. This is our key lemma for counting integer vectors after partitioning them by sparsity.
\begin{lemma}[Lemma 3.3 in \cite{MR4282089}]\label{num}: Let $p$ be an odd prime and let $k \in \mathbb{N}, s_1 \geq s_2 \in[n], t \in[p]$. Let
\begin{align*}
\boldsymbol{B}_{k, s_2, \geq t}^{s_1}(n):=\left\{\boldsymbol{a} \in \mathbb{F}_p^n:|\operatorname{supp}(\boldsymbol{a})| \geq s_1, \forall \boldsymbol{b} \subset \boldsymbol{a} \text { s.t. }|\operatorname{supp}(\boldsymbol{b})| \geq s_2\right.\; 
\text { we have } \left.R_k^*(\boldsymbol{b}) \geq t \cdot \frac{2^{2 k} \cdot|\boldsymbol{b}|^{2 k}}{p}\right\}
\end{align*}
Then,
\begin{align*}
\left|\boldsymbol{B}_{k, s_2, \geq t}^{s_1}(n)\right| \leq(200)^n\left(\frac{s_2}{s_1}\right)^{2 k-1} p^n t^{-n+s_2}.
\end{align*}
where $\boldsymbol{b} \subset \boldsymbol{a}$ means $\boldsymbol{b}$ can be obtained by projection from the coordinates of $\boldsymbol{a}$.
\end{lemma}
\par
We now show Basak and Rudelson's result \cite{MR3620692} about the spectral norm of sparse random matrices; we can then directly establish the relevant result about sparse random matrices with $\mu_n-$lazy entries $(\mu_{n}=\Omega(\frac{1}{n^{0.45}}))$.

\begin{lemma}[Theorem 1.7 in \cite{MR3620692}]\label{norm}
There exists $C_0 \geq 1$ such that the following holds. Let $n \in \mathbb{N}$ and $\mu_{n} \in(0,1]$ be such that $\mu_{n} \geq C_0 \frac{\log n}{n}$. Then there exist positive constants $C_{2.4}, c_{2.4}$, depending on $\mu_{n}$, so that
\begin{align*}
\textbf{P}\left(\left\|M_n\right\| \geq C_{2.4} \sqrt{n \mu_{n}}\right) \leq \exp \left(-c_{2.4} n \mu_{n}\right).
\end{align*}
\end{lemma}
The last lemma shows that, with very high probability, the image of a fixed unit vector under $M_n$ does not have norm $o(\sqrt{n\mu_{n}})$.
\par
\begin{lemma}[Lemma 4.1 in \cite{MR2663633}]\label{norm2}
Fix a unit vector $\boldsymbol{a}=\left(a_1, \ldots, a_n\right) \in \mathbb{S}^{n-1}$, there exist positive absolute constants $C_{2.5}$ and $c_{2.5}$ such that 
\begin{align*}
\textbf{P}\left(\|M_n\boldsymbol{a}\|_2 \leq C_{2.5} \sqrt{n \mu_{n}}\right) \leq \exp \left(-c_{2.5} n \mu_{n}\right).
\end{align*}
\end{lemma}
In Lemma \ref{norm2}, the conditions given by G\"{o}tze and Tikhomirov for $\mu_{n}$ are $\mu_{n}^{-1}\gg o(\frac {n} {ln^2n})$, $\mu_{n}=\Omega(\frac{1}{n^{0.45}})$ is clearly satisfied.

\section{Proof of main result}\label{sec3}
In this section, we are devoted to proving our main result. Throughout this section, we will take $\mu_{n}=\Omega (\frac{1}{n^{0.45}})$, $\alpha:=n^{1 / 4}$, $C^*_{2.4}:=\max\{C_{2.4},1\}$ and $\gamma:=C_{2.5} /\left(r C^*_{2.4}\right)$ (where r is a constant for the purpose of $\gamma \leq 1 / 100$ and $r\ge100$ ). Moreover, we can observe that $\textbf{P}\left(s_n\left(M_n\right) \leq \eta\right) \lesssim \frac{\eta n^{3 / 2}}{\sqrt{\mu_{n}}}$ holds trivially for $\eta \geq n^{-3 / 2}\sqrt{\mu_{n}}$, thus, we will henceforth assume that $2^{-n^{0.0001}} \leq \eta<n^{-3 / 2}\sqrt{\mu_{n}}$.
\par
Based on the magnitude of the least common denominator (LCD) of the vectors compared to $n^{3/4}\sqrt{\mu_{n}}\cdot\eta^{-1}$, we decompose the unit sphere $\mathbb{S}^{n-1}$ into $\Gamma^1(\eta) \cup \Gamma^2(\eta)$, where
\begin{align*}
\Gamma^1(\eta):=\left\{\boldsymbol{a} \in \mathbb{S}^{n-1}: \operatorname{LCD}_{\alpha, \gamma}(\boldsymbol{a}) \geq  n^{3 / 4}\sqrt{\mu_{n}} \cdot \eta^{-1}\right\},
\end{align*}
and $\Gamma^2(\eta):=\mathbb{S}^{n-1} \backslash \Gamma^1(\eta)$, we have
\begin{align*}
\textbf{P}\left(s_n\left(M_n\right) \leq \eta\right) \leq  \textbf{P}\left(\exists \boldsymbol{a} \in \Gamma^1(\eta):\left\|M_n \boldsymbol{a}\right\|_2 \leq \eta\right) +\textbf{P}\left(\exists \boldsymbol{a} \in \Gamma^2(\eta):\left\|M_n \boldsymbol{a}\right\|_2 \leq \eta\right).
\end{align*}
Therefore, Theorem \ref{main} follows from the following two propositions and the union bound.
\begin{proposition}\label{pro1}
$\textbf{P}\left(\exists \boldsymbol{a} \in \Gamma^1(\eta):\left\|M_n \boldsymbol{a}\right\|_2 \leq \eta\right) \lesssim \frac{\eta n^{3 / 2}}{\sqrt{\mu_{n}}}+n \exp (-2\sqrt{n}\mu_{n})$.
\end{proposition}

\begin{proposition}\label{pro2}
$\textbf{P}\left(\exists \boldsymbol{a} \in \Gamma^2(\eta):\left\|M_n \boldsymbol{a}\right\|_2 \leq \eta\right) \lesssim \exp \left(-c_{3.2} n\mu_{n}\right)$.
\end{proposition}

\begin{proof}[Proof of Proposition \ref{pro1}]
For every $M_n$ satisfying the above assumptions, since $M_n$ and $M_n^T$ have the same singular values, the event $s_n(M_n)\leq\eta$ implies the existence of a unit vector $\boldsymbol a'\in\mathbb R^n$ for which
\begin{align*}
\|\boldsymbol a'{}^T M_n\|_2 \leq\eta.
\end{align*}
For each fixed $M_n$, we define the unit vector $\boldsymbol a'\in\mathbb R^n$ satisfying the above condition to be $\boldsymbol a'_{M_n}$. Since the matrix rows are independent of each other, this leads to a partition of the space of all the above types of matrices with the smallest singular value at most $\eta$. Then, by taking a union bound, 
\begin{align*}
\textbf{P}\left(\exists \boldsymbol{a} \in \Gamma^1(\eta):\left\|M_n \boldsymbol{a}\right\|_2 \leq \eta\right)\leq n \textbf{P}\left(\exists\,\boldsymbol a\in\Gamma^1(\eta):
\|M_n\boldsymbol a\|_2\le\eta \wedge
\|\boldsymbol a'_{M_n}\|_\infty=|a'_n|\right),
\end{align*}
it suffices to show the following:
\begin{align}\label{pro30}
\textbf{P}\left(\exists\,\boldsymbol a\in\Gamma^1(\eta):
\|M_n\boldsymbol a\|_2\le\eta \wedge
\|\boldsymbol a'_{M_n}\|_\infty=|a'_n|\right)
\lesssim\eta\sqrt {\frac{n}{\mu_{n}}} +e^{-2\sqrt n\mu_{n}}.
\end{align}
To establish this, reveal the first \(n-1\) rows \(X_1,\dots,X_{n-1}\).  If there are some $\boldsymbol a\in\Gamma^1(\eta)$ satisfies $\|M_n\boldsymbol a\|_2\le\eta$, then based solely on \(X_1,\dots,X_{n-1}\) one can select a vector \(\boldsymbol y\in\Gamma^1(\eta)\) such that
\[
\left(\sum_{i=1}^{n-1}(X_i\cdot\boldsymbol y)^2\right)^{1/2}
\le\eta.
\]
Thus, for any vector $\boldsymbol{w}^{\prime} \in \mathbb{S}^{n-1}$ with $w_n^{\prime} \neq 0$, we can rewrite $X_n$ with $X_1,...,X_{n-1}$:
\begin{align*}
X_n=\frac{1}{w_n^{\prime}}\left(\boldsymbol{u}-\sum_{i=1}^{n-1} w_i^{\prime} X_i\right),
\end{align*}
where $\boldsymbol{u}:=\boldsymbol{w}^{\prime T} M_n$. Indeed, since $X_n \cdot \boldsymbol{y}$ is irrelevant to how $\boldsymbol{w}^{\prime}$ is chosen,
thus, for the event $\left\{s_n\left(M_n\right) \leq \eta\right\} \wedge\left\{\left\|\boldsymbol{a}^{\prime}{ }_{M_n}\right\|_{\infty}=\left|a_n^{\prime}\right|\right\}$ to occur, we must necessarily have
\begin{align}\label{pro31}
\left|X_n \cdot \boldsymbol{y}\right|  =\inf _{\boldsymbol{w}^{\prime} \in \mathbb{S}^{n-1}, w_n^{\prime} \neq 0} \frac{1}{\left|w_n^{\prime}\right|}\left|\boldsymbol{u} \cdot \boldsymbol{y}-\sum_{i=1}^{n-1} w_i^{\prime} X_i \cdot \boldsymbol{y}\right|,
\end{align}
then, using the inequality Cauchy-Schwarz inequality for the right-hand side of (\ref{pro31}) and setting $\boldsymbol{w}^{\prime}=\boldsymbol{a}^{\prime}{ }_{M_n}$, we have
\begin{align*}
\left|X_n \cdot \boldsymbol{y}\right|&\leq\frac{1}{\left|a_n^{\prime}\right|}\left(\left\|\boldsymbol{a}_{M_n}^{\prime T} M_n\right\|_2\|\boldsymbol{y}\|_2+\left\|\boldsymbol{a}_{M_n}^{\prime}\right\|_2\left(\sum_{i=1}^{n-1}\left(X_i \cdot \boldsymbol{y}\right)^2\right)^{1 / 2}\right)\\
&\leq\eta\sqrt{n}\left(\|\boldsymbol{y}\|_2+\left\|\boldsymbol{a}_{M_n}^{\prime}\right\|_2\right) \leq 2 \eta \sqrt{n},
\end{align*}
where the second inequality is due to the assumption that event $\{s_n\left(M_n\right) \leq \eta\} \wedge\{\left\|\boldsymbol{a}^{\prime}_{M_n}\right\|_{\infty}=\left|a_n^{\prime}\right|\}$ occurs. It follows, by definition, that the probability in (\ref{pro30}) is bounded by $\mathcal{L}\left(X_n \cdot \boldsymbol{y}, 2 \eta \sqrt{n}\right)$, and using Lemma \ref{lcd}, by
\begin{align*}
\mathcal{L}\left(X_n \cdot \boldsymbol{y}, 2 \eta \sqrt{n}\right) \lesssim \eta \sqrt{\frac{n}{\mu_{n}}}+\exp (-2\sqrt{n}\mu_{n}).
\end{align*}
which completes the proof.
\end{proof}
For the reader's convenience, we use the following three subsections to Proposition \ref{pro2}.
\subsection{Reduction to integer vectors}
Since the vectors in $\Gamma^2(\eta)$ are still essentially vectors on the unit sphere $\mathbb{S}^{n-1}$, this creates difficulties for our following study. Here, we present the first key step, which is the efficient transfer from vectors on the unit sphere to integer vectors.
\begin{proposition}\label{pro421}
With the notation above, we have
\begin{align*}
\textbf{P}(\exists \boldsymbol{a} \in \Gamma^2(\eta)&:\left\|M_n \boldsymbol{a}\right\|_2 \leq \eta) \\
\lesssim e^{{-c_{2.4}}n\mu_{n}}+\textbf{P}&(\exists \boldsymbol{w} \in\left(\mathbb{Z}^n\backslash\{\mathbf{0}\}\right)\cap\left[-2 \eta^{-1} n^{3 / 4}\sqrt{\mu_{n}}, 2 \eta^{-1} n^{3 / 4}\sqrt{\mu_{n}}\right]^n: \\
& \left.\left\|M_n \boldsymbol{w}\right\|_2 \leq \min \left\{4 \gamma C^*_{2.4} \sqrt{n\mu_{n}}\|\boldsymbol{w}\|_2, 2 C^*_{2.4} \alpha \sqrt{n\mu_{n}}\right\}\right).
\end{align*}
\end{proposition}
\begin{proof}
Let $\boldsymbol{a} \in \Gamma^2(\eta)$. Then, recall the definition of $\operatorname{LCD}$, there exists some $0< \operatorname{LCD}_{\alpha, \gamma}(\boldsymbol{a})\leq\theta  \leq n^{3 / 4}\sqrt{\mu_{n}} \eta^{-1}$ and some $\boldsymbol{w} \in \mathbb{Z}^n \backslash\{\mathbf{0}\}$ such that
\begin{align*}
\|\theta \boldsymbol{a}-\boldsymbol{w}\|_2 \leq \min \{\gamma \theta, \alpha\}.
\end{align*}
in particular, $\theta(1 - \gamma)\leq\|\boldsymbol{w}\|_2\leq \theta(1 + \gamma)$ .
\par
Note that, since $\theta(1 \pm \gamma) \leq 2 \theta$, it follows that each coordinate of $\boldsymbol{w}$ is bounded in absolute value by $2 \theta \leq 2 \eta^{-1} n^{3 / 4}\sqrt{\mu_{n}}$, i.e.,
\begin{align*}
\boldsymbol{w} \in \left(\mathbb{Z}^n \backslash\{\mathbf{0}\}\right) \cap\left[-2 \eta^{-1} n^{3 / 4}\sqrt{\mu_{n}}, 2 \eta^{-1} n^{3 / 4}\sqrt{\mu_{n}}\right]^n .
\end{align*}
If $\left\|M_n \boldsymbol{a}\right\|_2 \leq \eta$, it follows from the triangle inequality that
\begin{align*}
\left\|M_n \boldsymbol{w}\right\|_2 & =\left\|M_n(\boldsymbol{w}-\theta \boldsymbol{a})+M_n(\theta \boldsymbol{a})\right\|_2 \\
& \leq\left\|M_n\right\| \cdot\|\theta \boldsymbol{a}-\boldsymbol{w}\|_2+\theta \cdot\left\|M_n \boldsymbol{a}\right\|_2 \\
& \leq C_{2.4} \sqrt{n\mu_{n}} \cdot \min \{\gamma \theta, \alpha\}+\theta \eta\\
& \leq 2 C^*_{2.4} \sqrt{n\mu_{n}} \cdot \min \{\gamma \theta, \alpha\},
\end{align*}
where the second inequality is due to the bounds on the spectral norms of $M_n$ given by Lemma \ref{norm}:
\begin{align*}
\textbf{P}\left(\left\|M_n\right\| \geq C_{2.4} \sqrt{n \mu_{n}}\right) \leq \exp \left(-c_{2.4} n \mu_{n}\right), 
\end{align*}
and the last inequality follows since $\eta \ll \gamma \sqrt{n\mu_n}$ and $\theta \eta\leq n^{3 / 4}\sqrt{\mu_{n}}= \sqrt{n\mu_n} \alpha$. With the following simple case study, we can draw the desired conclusion.

$\textbf{Case I}$: $\gamma \theta \leq \alpha$. In this case, $\boldsymbol{w}$ satisfies
\begin{align*}
\left\|M_n \boldsymbol{w}\right\|_2 \leq 2 \gamma C^*_{2.4} \sqrt{n\mu_{n}} \theta \leq \min \left\{4 \gamma C^*_{2.4} \sqrt{n\mu_{n}}\|\boldsymbol{w}\|_2, 2 C^*_{2.4} \alpha \sqrt{n\mu_{n}}\right\},
\end{align*}
where the last inequality uses $\theta / 2 \leq \theta(1-\gamma) \leq\|\boldsymbol{w}\|_2$ and $\gamma \theta \leq \alpha$.

$\textbf{Case II}$: $\gamma \theta>\alpha$. In this case, $\boldsymbol{w}$ satisfies
\begin{align*}
\left\|M_n \boldsymbol{w}\right\|_2 & \leq 2 C^*_{2.4} \alpha \sqrt{n\mu_{n}} \leq \min \left\{2 C^*_{2.4} \gamma^{-1} \alpha \gamma \sqrt{n\mu_{n}}, 2 C^*_{2.4} \alpha \sqrt{n\mu_{n}}\right\} \\
& \leq \min \left\{4 \gamma C^*_{2.4} \sqrt{n\mu_{n}}\|\boldsymbol{w}\|_2, 2 C^*_{2.4} \alpha \sqrt{n\mu_{n}}\right\},
\end{align*}
where the last inequality uses
\begin{align*}
\|\boldsymbol{w}\|_2 \geq \theta(1-\gamma) \geq \gamma^{-1} \cdot(\gamma \theta) / 2 \geq \gamma^{-1} \alpha / 2 .
\end{align*}
\end{proof}
\par
The following study is to partition the integer vectors according to the size of their support set, and interestingly, through subsequent proofs, we find that we need to use a sparser partition of the integer vectors to study the smallest singular value of sparse matrices.

\subsection{Dealing with sparse integer vectors}
Furthermore, we need to transfer the vector to a prime range, and for a sparse vector $\boldsymbol{w}$, we have a simpler union bound. The goal of this subsection is to prove the following lemma, which follows from Lemma \ref{norm2}. Throughout this subsection and the next one, $p=2^{n^{0.001}}$ is a prime. Note that $p \gg \eta^{-1} n^{3 / 4}\gg \eta^{-1} n^{3 / 4}\sqrt{\mu_{n}}$, this satisfies the assumptions of the previous subsection.

\begin{lemma}\label{422}
\begin{align*}
\textbf{P}\left(\exists \boldsymbol{w} \in\left(\mathbb{Z}^n \backslash\{\mathbf{0}\}\right) \cap[-p, p]^n,|\operatorname{supp}(\boldsymbol{w})| \leq n^{0.99}\mu_{n}:\left\|M_n \boldsymbol{w}\right\|_2 \leq\right. & \left.4 \gamma C^*_{2.4} \sqrt{n\mu_{n}}\|\boldsymbol{w}\|_2\right) \\
\lesssim & \exp \left(-\frac{c_{2.5} n\mu_{n}}{2}\right).
\end{align*}
\end{lemma}
\begin{proof}
The number of vectors $\boldsymbol{w} \in\left(\mathbb{Z}^n \backslash\{\mathbf{0}\}\right) \cap[-p, p]^n$ with support of size no more than $n^{0.99}\mu_{n}$ is at most
\begin{align*}
\binom{n}{n^{0.99}\mu_{n}}(3 p)^{n^{0.99}\mu_{n}} \ll 2^{n^{0.992}\mu_{n}}.
\end{align*}
By Lemma \ref{norm2}, for any such vector,
\begin{align*}
\textbf{P}\left(\left\|M_n \boldsymbol{w}\right\|_2 \leq 4 \gamma C^*_{2.4} \sqrt{n\mu_{n}}\|\boldsymbol{w}\|_2\right) & \leq \textbf{P}\left(\left\|M_n \boldsymbol{w}\right\|_2 \leq C_{2.5} \sqrt{n\mu_{n}}\|\boldsymbol{w}\|_2\right) \\
& \lesssim \exp \left(-c_{2.5} n\mu_{n}\right).
\end{align*}
Therefore, the union bound gives the desired conclusion.
\end{proof}

\subsection{Dealing with non-sparse integer vectors}
Throughout this subsection, we fix
\begin{align*}
k=n^{0.01}, \quad s_1=s_2=n^{0.99}\mu_{n}.
\end{align*}
It remains to deal with integer vectors with support of size at least $n^{0.99}\mu_{n}$; this completes the division of the previous section. Formally, let
\begin{align*}
\boldsymbol{W}:=\left\{\boldsymbol{w} \in\left(\mathbb{Z}^n \backslash\{\mathbf{0}\}\right) \cap\left[-\eta^{-4}, \eta^{-4}\right]^n:|\operatorname{supp}(\boldsymbol{w})| \geq n^{0.99}\mu_{n}\right\} .
\end{align*}
Since $\eta \leq n^{-3 / 2}\sqrt{\mu_{n}}$, $\left[-\eta^{-4}, \eta^{-4}\right]^n$ can cover the space given in Proposition \ref{pro421}, and the following proposition suffices to prove Proposition \ref{pro2}.

\begin{proposition}\label{pro45}
$\textbf{P}\left(\exists \boldsymbol{w} \in \boldsymbol{W}:\left\|M_n \boldsymbol{w}\right\|_2 \leq 2 C^*_{2.4} n^{3 / 4}\sqrt{\mu_{n}}\right) \lesssim n^{-0.01 n}$.
\end{proposition}
We begin with a refined division of $\boldsymbol{W}$, since $p\gg\eta^{-4}$, the natural map
\begin{align*}
\iota: \boldsymbol{W} \rightarrow \mathbb{F}_p^n,\quad \iota(\boldsymbol{w})=\boldsymbol{w}\bmod p
\end{align*}
is injective. To make it more convenient to use Theorem \ref{Ra1} and Lemma \ref{num} established in Section \ref{sec2}, we can make the following definition.
\par
For an integer $t \in[p]$, let
\begin{align*}
\boldsymbol{W}_t:=\left\{\boldsymbol{w} \in \boldsymbol{W}: \iota(\boldsymbol{w}) \in \boldsymbol{B}_{k, s_2, \geq t-1}^{s_1}(n) \backslash \boldsymbol{B}_{k, s_2, \geq t}^{s_1}(n)\right\} .
\end{align*}
We will need the following two lemmas.
\begin{lemma}\label{lem47}
There exists an absolute constant $C_{3.2}>1$ such that, for our choice of parameters, if $\boldsymbol{w} \in \boldsymbol{W}_t$, then
\begin{align*}
\rho^{(\mu_{n})}(\boldsymbol{w}) \leq \frac{C_{3.2}}{p}\left(\frac{t}{n^{0.48}\mu_{n}}+1\right).
\end{align*}
\end{lemma}
\begin{proof}
Since $\rho^{(\mu_{n})}(\boldsymbol{w}) \leq \rho^{(\mu_{n})}_{\mathbb{F}_p}(\iota(\boldsymbol{w}))=: \rho^{(\mu_{n})}_{\mathbb{F}_p}(\boldsymbol{w})$, it suffices to prove the statement for the latter quantity. Indeed, since $\boldsymbol{w} \notin \boldsymbol{B}_{k, s_2, \geq t}^{s_1}(n)$, there exists some $\boldsymbol{b} \subset \boldsymbol{w}$ such that $|\operatorname{supp}(\boldsymbol{b})| \geq s_2$ and
\begin{align*}
R_k^*(\boldsymbol{b}) \leq t \cdot \frac{2^{2 k} \cdot|\boldsymbol{b}|^{2 k}}{p}.
\end{align*}
Furthermore, for our choice of parameters, we have
\begin{align*}
\left(40 k^{0.99} n^{1.01}\right)^k \ll \frac{2^{2 k} s_2^{2 k}}{p} \leq t \cdot \frac{2^{2 k} \cdot|\boldsymbol{b}|^{2 k}}{p}.
\end{align*}
Hence, applying Hal\'{a}sz's inequality (Theorem \ref{Hal}) to the $|\boldsymbol{b}|$-dimensional vector $\boldsymbol{b}$ with
\begin{align*}
M=n^{0.96}\mu_{n},
\end{align*}
(note that this choice of $M$ satisfies the conditions $30 M \leq s_2 \leq|\boldsymbol{\textbf{supp}}(\boldsymbol{b})|$ and $80 k M \leq s_2 \leq|\boldsymbol{b}|$ needed to apply Hal\'{a}sz's inequality), and by a simple observation $\rho^{(\mu_{n})}_{\mathbb{F}_p}(\boldsymbol{w}) \leq \rho^{(\mu_{n})}_{\mathbb{F}_p}(\boldsymbol{b})$, we have

\begin{align*}
\rho^{(\mu_n)}_{\mathbb{F}_p}(\boldsymbol{w}) & \lesssim \frac{1}{p}+\frac{t \cdot \frac{2^{2 k} \cdot|\boldsymbol{b}|^{2 k}}{p}}{2^{2 k}|\boldsymbol{b}|^{2 k} n^{0.48}\mu_{n}}+e^{-n^{0.96}\mu_{n}^2} \\
& \lesssim \frac{1}{p}\left(\frac{t}{n^{0.48}\mu_{n}}+1\right).
\end{align*}
\end{proof}
\begin{lemma}\label{lem48}
For our choice of parameters, and for $\sqrt{p} \leq t \leq p$,
\begin{align*}
\left|\boldsymbol{W}_t\right| \leq(300)^n\left(\frac{p}{t}\right)^n.
\end{align*}
\end{lemma}
\begin{proof}
By definition, any $\boldsymbol{w} \in \boldsymbol{W}_t$ satisfies $\iota(\boldsymbol{w}) \in \boldsymbol{B}_{k, s_2, \geq t-1}^{s_1}(n)$. Since the map $\iota$ is injective, we get
\begin{align*}
\left|\boldsymbol{W}_t\right| \leq \left|\boldsymbol{B}_{k, s_2, \geq t-1}^{s_1}(n)\right| \leq (200)^n\left(\frac{p}{t-1}\right)^n p^{s_2} \leq(300)^n\left(\frac{p}{t}\right)^n,
\end{align*}
where the second inequality holds by Lemma \ref{num}.
\begin{align*}
(200)^n\left(\frac{p}{t}\right)^n p^{s_2} \leq(300)^n\left(\frac{p}{t}\right)^n.
\end{align*}
\end{proof}

Finally, we are in a position to prove Proposition \ref{pro45}. As discussed earlier, Lemma \ref{422} and Proposition \ref{pro45} give a further estimate for Proposition \ref{pro421}, thereby completing the proof of Proposition \ref{pro2}. Furthermore, Proposition \ref{pro1} and Proposition \ref{pro2} complete the proof of the main result Theorem \ref{main}.

\begin{proof}[\bf Proof of Proposition \ref{pro45}]
We begin by determining a lower bound for t. Note that every $\boldsymbol{w} \in \boldsymbol{W}$ has
\begin{align*}
\rho^{(\mu_{n})}(\boldsymbol{w}) \geq \eta^4 n^{-1} / 3.
\end{align*}
Since for any such vector $\boldsymbol{w}, \sum_{i=1}^n \xi_i w_i$ can only take on at most $3 n \eta^{-4}$ values, note that $\eta^4 n^{-1} / 3 \gg 1 / \sqrt{p}$, by Lemma \ref{lem47}, it is easy to see that $\boldsymbol{W}_t=\emptyset$ for all $t \leq \sqrt{p}$.
\par
Since $\boldsymbol{W}_t \subseteq \mathbb{Z}^n$ is a set of nonzero integer vectors. Then, for $C(n) \geq 1$,
\begin{align*}
\textbf{P}&\left(\exists \boldsymbol{w} \in \boldsymbol{W}_t:\left\|M_n \boldsymbol{w}\right\|_2 \leq C(n) \sqrt{n}\right)\\
& \leq \sum_{\boldsymbol{z} \in \mathbb{Z}^n \cap B(0, C(n) \sqrt{n})} \textbf{P}\left(\exists \boldsymbol{w} \in \boldsymbol{W}_t: M_n \boldsymbol{w}=\boldsymbol{z}\right) \\
& \leq\left|\mathbb{Z}^n \cap B(0, C(n) \sqrt{n})\right| \cdot \sup _{\boldsymbol{z} \in \mathbb{Z}^n \cap B(0, C(n) \sqrt{n})} \textbf{P}\left(\exists \boldsymbol{w} \in \boldsymbol{W}_t: M_n \boldsymbol{w}=\boldsymbol{z}\right) \\
& \leq(100 C(n))^n \cdot \sup _{\boldsymbol{z} \in \mathbb{Z}^n} \textbf{P}\left(\exists \boldsymbol{w} \in \boldsymbol{W}_t: M_n \boldsymbol{w}=\boldsymbol{z}\right),
\end{align*}
where the last inequality is a classical result for higher-dimensional balls covering integer points, i.e. 
\begin{align*}
\left|\mathbb{Z}^n \cap B(0, C(n) \sqrt{n})\right|\le vol(B(0, (C(n)+1) \sqrt{n})).
\end{align*}
\par
Let $C(n)=2 C^*_{2.4} n^{1 / 4}\sqrt{\mu_{n}}$, it follows from Lemma \ref{lem47} and Lemma \ref{lem48} that for all $t \geq \sqrt{p}$,
\\
$\textbf{P}\left(\exists \boldsymbol{w} \in \boldsymbol{W}_t:\left\|M_n \boldsymbol{w}\right\|_2 \leq 2 C^*_{2.4} n^{3 / 4}\sqrt{\mu_{n}}\right)$ can be bounded as follows
\begin{align*}
\left(200 C^*_{2.4} n^{1 / 4}\sqrt{\mu_{n}}\right)^n\left|\boldsymbol{W}_t\right|\left(\frac{2 C_{3.2} t}{p n^{0.48}\mu_{n}}\right)^n & \leq\left(200 C^*_{2.4} n^{1 / 4}\sqrt{\mu_{n}}\right)^n(300)^n\left(\frac{p}{t}\right)^n\left(\frac{2 C_{3.2} t}{p n^{0.48}\mu_{n}}\right)^n \\
& \ll n^{-0.01 n}.
\end{align*}
Finally, taking the union bound over integers $t \in[\sqrt{p}, p]$ completes the proof.
\par
Now, everything is ready to prove the main result of the paper.
\end{proof}
\begin{proof}[\bf Proof of Theorem \ref{main}]
As discussed in Proposition \ref{pro421}, we have
\begin{align*}
\textbf{P}(\exists \boldsymbol{a} \in \Gamma^2(\eta)&:\left\|M_n \boldsymbol{a}\right\|_2 \leq \eta) \\
\lesssim e^{{-c_{2.4}}n\mu_{n}}+\textbf{P}&(\exists \boldsymbol{w} \in\left(\mathbb{Z}^n\backslash\{\mathbf{0}\}\right)\cap\left[-2 \eta^{-1} n^{3 / 4}\sqrt{\mu_{n}}, 2 \eta^{-1} n^{3 / 4}\sqrt{\mu_{n}}\right]^n: \\
& \left.\left\|M_n \boldsymbol{w}\right\|_2 \leq \min \left\{4 \gamma C^*_{2.4} \sqrt{n\mu_{n}}\|\boldsymbol{w}\|_2, 2 C^*_{2.4} \alpha \sqrt{n\mu_{n}}\right\}\right),
\end{align*}
where
\begin{align*}
\textbf{P}&(\exists \boldsymbol{w} \in\left(\mathbb{Z}^n\backslash\{\mathbf{0}\}\right)\cap\left[-2 \eta^{-1} n^{3 / 4}\sqrt{\mu_{n}}, 2 \eta^{-1} n^{3 / 4}\sqrt{\mu_{n}}\right]^n:\left.\left\|M_n \boldsymbol{w}\right\|_2 \leq \min \left\{4 \gamma C^*_{2.4} \sqrt{n\mu_{n}}\|\boldsymbol{w}\|_2, 2 C^*_{2.4} \alpha \sqrt{n\mu_{n}}\right\}\right)\\
&\leq \textbf{P}(\exists \boldsymbol{w} \in\left(\mathbb{Z}^n\backslash\{\mathbf{0}\}\right)\cap\left[-2 \eta^{-1} n^{3 / 4}\sqrt{\mu_{n}}, 2 \eta^{-1} n^{3 / 4}\sqrt{\mu_{n}}\right]^n,|\operatorname{supp}(\boldsymbol{w})|\leq n^{0.99}\mu_n:\left\|M_n \boldsymbol{w}\right\|_2 \leq  4 \gamma C^*_{2.4} \sqrt{n\mu_{n}}\|\boldsymbol{w}\|_2)\\
&\quad +\textbf{P}(\exists \boldsymbol{w} \in\left(\mathbb{Z}^n\backslash\{\mathbf{0}\}\right)\cap\left[-2 \eta^{-1} n^{3 / 4}\sqrt{\mu_{n}}, 2 \eta^{-1} n^{3 / 4}\sqrt{\mu_{n}}\right]^n,|\operatorname{supp}(\boldsymbol{w})|\geq n^{0.99}\mu_n:\left\|M_n \boldsymbol{w}\right\|_2 \leq  2 C^*_{2.4} \alpha \sqrt{n\mu_{n}})\\
&\leq \exp{\left(-\frac{c_{2.5}n\mu_n}{2}\right)}+n^{-0.01n}.
\end{align*}
Thus, we have completed the proof of Proposition \ref{pro2}:
\begin{align*}
\textbf{P}(\exists \boldsymbol{a} \in \Gamma^2(\eta)&:\left\|M_n \boldsymbol{a}\right\|_2 \leq \eta)\lesssim\exp\left(-c_{3.2}n\mu_n\right).
\end{align*}
Furthermore, from Proposition \ref{pro1} and Proposition \ref{pro2}, we obtain
\begin{align*}
\textbf{P}\left(s_n\left(M_n\right) \leq \eta\right) &\leq \textbf{P}\left(\exists \boldsymbol{a} \in \Gamma^1(\eta):\left\|M_n \boldsymbol{a}\right\|_2 \leq \eta\right) +\textbf{P}\left(\exists \boldsymbol{a} \in \Gamma^2(\eta):\left\|M_n \boldsymbol{a}\right\|_2 \leq \eta\right)\\
&\lesssim \frac{\eta n^{3 / 2}}{\sqrt{\mu_{n}}}+n \exp (-2\sqrt{n}\mu_{n})+\exp\left(-c_{3.2}n\mu_n\right)\lesssim\frac{\eta n^{3 / 2}}{\sqrt{\mu_{n}}},
\end{align*}
since $\eta\ge2^{-n^{0.0001}}$, the last inequality holds for all sufficiently large n. This completes the proof.
\end{proof}

\section{Discussion}\label{sec4}
In the related studies on the smallest singular value of random matrices, there are fewer studies on sparse random matrices. Like Rudelson and Vershynin, we use the key notion of LCD of a vector for a simple partition of the $\mathbb{S}^{n-1}$ and introduce the counting method developed by Ferber, Jain, Luh and Samotij \cite{MR4273471} to the study of sparse vectors. Then, we extend the study of fixed $\mu$ in Tao and Vu \cite{MR2480613} to $\mu_n$, which tends to zero.
\par
In future work, there are two problems to study. The first problem is to refine the LCD to partition $\mathbb{S}^{n-1}$. Litvak and Tikhomirov \cite{MR4402560}, in solving the problem of the smallest singular value of sparse random matrices with entries 0/1, introduced the degree of unstructuredness (u-degree) to better deal with the problems posed by sparsity. Inspired by this, finding a more refined tool to replace LCD in the study of sparse matrices may lead to better estimates.
\par
Another problem is the study of symmetric sparse matrices, where the challenge lies in the fact that we cannot directly use row independence to classify the study of random matrices as the study of random vectors. The recent studies of Campos, Jenssen, Michelen and Sahasrabudhe \cite{MR4810062} on the smallest singular value of symmetric random matrices have given us some approaches which will help us in our future studies of symmetric sparse matrices.

\printbibliography

\end{document}